\documentclass[reqno]{amsart}
\usepackage{amssymb,amsmath,amsthm,amsfonts,latexsym,booktabs,qtree,tikz,comment,url}
\usepackage[all]{xy}
\usepackage{multicol}
\usepackage[hidelinks]{hyperref}
\usepackage[pagewise]{lineno}
\usepackage{cite}
\theoremstyle{plain}
\newtheorem{definition}{Definition}
\newtheorem{theorem}{Theorem}
\newtheorem{corollary}[theorem]{Corollary}
\newtheorem{lemma}[theorem]{Lemma}

\newcommand{\alE}{\mathcal{E}}		
\newcommand{\alD}{\mathcal{D}_t}	
\newcommand{\alJ}{\mathcal{J}} 		
\newcommand{\mM}{\mathcal{M}}		
\newcommand{\mN}{\mathcal{N}}		
\newcommand{\F}{\mathbb{F}}  		
\newcommand{\Fd}{\F\cdot}    		%
\newcommand{\pa}{|a|}				
\newcommand{\pb}{|b|}				
\newcommand{\pc}{|c|}				
\newcommand{\pd}{|d|}				

\newcommand{\redt}{\mathrm{Reg}\,\alD} 					
\newcommand{\zfmc}[2]{\mathcal{Z}^2_{#1}(\alD,\,#2)}	
\newcommand{\bfmc}[2]{\mathcal{B}^2_{#1}(\alD,\,#2)} 	
\newcommand{\hfmc}[2]{\mathcal{H}^2_{#1}(\alD,\,#2)} 	
\newcommand{\am}{^{(+)}}
\newcommand{\suMnm}{\mathcal{M}_{n|m}(\F)\am}
\newcommand{\osp}{\textnormal{osp}}
\newcommand{\josnm}{\alJ\osp_{n|2m}(\F)}
\newcommand{\we}{\widetilde{e}}         				
\newcommand{\wx}{\widetilde{x}}							
\newcommand{\wy}{\widetilde{y}}							
\newcommand{\fhi}[4]{h\left(\left(#1#2\right)#3,\,#4\right)+h(#1#2,\,#3)\cdot #4+\big(h(#1,\,#2)\cdot #3\big)\cdot #4} 				
\newcommand{\fhd}[4]{h(#1#2,\,#3#4)+h(#1,\,#2)\cdot (#3#4)+(#1#2)\cdot h(#3,\,#4)}													
\newcommand{\fhc}[4]{h((#1#2)#3,\,&#4) + h(#1#2,\,#3)\cdot #4 + (h(#1,\,#2)\cdot #3)\cdot#4\\
	+&(-1)^{|#2|(|#3|+|#4|)+|#3||#4|}\Big(h\big((#1#4)#3,\,#2\big)+h\big(#1#4,\,#3\big)\cdot#2+\big(h(#1,\,#4)\cdot#3\big)\cdot#2\Big)\\
	+&(-1)^{|#1|(|#2|+|#3|+|#4|)+|#3||#4|}\Big(h\big((#2#4)#3,\,#1\big)+h\big(#2#4,\,#3\big)\cdot#1+\big(h(#2,\,#4)\cdot#3\big)\cdot#1\Big)\\
	=h\big(#1#2,&\,#3#4\big)+h\big(#1,\,#2\big)\cdot(#3#4)+(#1#2)\cdot h\big(#3,\,#4\big)\\
	+&(-1)^{|#2||#3|}\Big(h\big(#1#3,\,#2#4\big)+h\big(#1,\,#3\big)\cdot(#2#4)+(#1#3)\cdot h\big(#2,\,#4\big)\Big)\\
	+&(-1)^{|#4|\big(|#2|+|#3|\big)}\Big(h\big(#1#4,\,#2#3\big)+h(#1,\,#4)\cdot(#2#3)+(#1#4)\cdot h\big(#2,\,#3\big)\Big)} 			
\allowdisplaybreaks

\begin{document}
	
\title[SCG of finite dimensional simple Jordan superalgebra $\mathcal{D}_{t}$]{Second cohomology group of the finite-dimensional simple Jordan superalgebra $\mathcal{D}_{t}$, $t\neq 0$}

\author[F. A.~G\'omez Gonz\'alez]{G\'omez Gonz\'alez, F. A. }
\address{Institute of Mathematics\\ University of Antioquia\\ Medell\'in\\ Colombia}
\email{faber.gomez@udea.edu.co}
\thanks{}

\author[J. A.~Ram\'irez Berm\'udez]{Ram\'irez Berm\'udez, J. A.}
\address{Institute of Mathematics\\ University of Antioquia\\ Medell\'in\\ Colombia}
\email{jalexander.ramirez@udea.edu.co}

\begin{abstract}
	The second cohomology group (SCG) of the Jordan superalgebra 
	$\mathcal{D}_{t}$, $t\neq 0$, is calculated by using the coefficients which appear in the regular superbimodule $\redt$. Contrary to the case of algebras, this group is nontrivial thanks to the non-splitting caused by the Wedderburn Decomposition Theorem \cite{Faber1}.
	
	First, to calculate the SCG of a Jordan superalgebra we use split-null extension of the Jordan superalgebra and the Jordan superalgebra representation. We prove conditions that satisfy the bilinear forms $h$ that determine the SCG in Jordan superalgebras. We use these to calculate the SCG for the Jordan superalgebra $\mathcal{D}_{t}$ , $t\neq 0$. 
	
	Finally, we prove that $\mathcal{H}^2(\mathcal{D}_{t}, \textrm{Reg}\mathcal{D}_{t})=0\oplus\F^2$, $t\neq 0$.
\end{abstract}

\keywords{Jordan superalgebra; second cohomology group; Wedderburn Principal Theorem; split null extension; regular superbimodule; decomposition theorem.}
\maketitle


\section{Introduction}
It is known that every separable, finite-dimensional Jordan algebra $\mathcal{J}$
has a trivial second cohomology group $\mathcal{H}^2(\mathcal{J},\mathcal{M})$, where $\mathcal{M}$ is a $\mathcal{J}$-bimodule, and the Wedderburn Principal Theorem (WPT)   holds for finite-dimensional Jordan algebras \cite{Jabcoson1968}. 

In 2018, the first author showed that the WPT does not hold for some finite Jordan superalgebras \cite{Faber1,FaberWPT,FaberRevo}, and therefore some finite Jordan superalgebras have nontrivial SCG. In particular, the Jordan superalgebra $\alD$ does not satisfy WPT  \cite{Faber1}. For this reason, we are interested in finding  the SCG for the finite Jordan superalgebra $\alD$. Some previous results about the SCG for alternative superalgebras were considered by Pisarenko and L\'opez-D\'iaz \cite{lopezdiaz,lopezdiaz1,pisarenko1993,pisarenkotesis,pisarenko1994}.

In this paper, Section \ref{preliminares} gives some preliminary results from the theory of Jordan superalgebras including those for the SCG in Jordan superalgebras. Then, Section \ref{SCG:h:functions}  establishes the conditions that satisfy the bilinear forms $h$ which define the SCG for a Jordan superalgebra. Finally, Section \ref{SCG:Dt} provides the calculation of SCG for the simple finite-dimensional Jordan superalgebra $\mathcal{D}_{t}$ using coefficients of $\textrm{Reg}\,\alD$. We prove that  $\hfmc{}{\redt}=0\oplus \F^2$, $t\neq 0$.



\section{Preliminaries}\label{preliminares}

Throughout the paper, all algebras are considered over an algebraically closed field  $\F$ of characteristic zero.

Note that $\alJ$ is said to be a $\textit{superalgebra}$,
if it is the direct sum $\alJ=\alJ_0\dot{+}\alJ_1$, 
where we denote the parity of $a$, $a\in \alJ_0 \dot\cup \alJ_1$ by $|a|=0,1$. Further, recall that a superalgebra $\alJ=\alJ_0\dot{+}\alJ_1$ is said to be a \textit{Jordan superalgebra}, if for every $a_i,a_j,a_k,a_l\in\alJ_0\dot\cup\alJ_{1}$ the superalgebra satisfies the superidentities
	\begin{equation}\label{sj1}
		a_ia_j=(-1)^{ij}a_ja_i,
	\end{equation}
	\begin{equation}\label{sj2}
	\begin{split}
	((a_ia_j)a_k)a_l+(-1)^{l(k+j)+jk}((a_ia_l)a_k)a_j+(-1)^{i(j+k+l)+kl}((a_ja_l)a_k)a_i\\
	=(a_ia_j)(a_ka_l)+(-1)^{l(k+j)}(a_ia_l)(a_ja_k)+(-1)^{kj}(a_ia_k)(a_ja_l).
	\end{split}
	\end{equation}

Throughout the paper, by $\dot{+}$ we denote the direct sum of vector spaces, by $+$ the sum  of vector spaces and by $\oplus$ a direct sum of superalgebras.

The classification of finite-dimensional simple Jordan superalgebras over an algebraically closed field of characteristic zero was given by Kac and Kantor \cite{Kac,Kantor1}.

A superbimodule $\mM=\mM_0\dot{+}\mM_1$ is called a \textit{Jordan $\alJ$-superbimodule},  
if the corresponding split-null extension $\mathcal{E}=\alJ\oplus\mM$ is a Jordan superalgebra. The multiplication in $\mathcal{E}$ is obtained from the multiplication 
in $\alJ$ and the action of $\alJ$ over $\mathcal{M}$, where $\mathcal{M}^2=0$. A \textit{regular $\alJ$-superbimodule}, denoted as $\mathrm{Reg}\,\alJ$, 
is defined  on the vector super-space $\alJ$ with an action coinciding with the multiplication in $\alJ$.

The classification of irreducible Jordan $\alD$-superbimodules over a finite dimensional, simple Jordan superalgebra $\alJ$ was given by Zelmanov and Martinez
\cite{MR2218992}.

Let $\alJ$ be a Jordan superalgebra and $\mM$, $\mN$ be the $\alJ$-superbimodules, then a linear mapping $\alpha: \mM\rightarrow\mN$ is called a \textit{homomorphism of superbimodule of degree $j$}, if the mapping is homogeneous of degree $j$, i.e.,  
$\alpha(\mathcal{M}_i)\subseteq\mathcal{N}_{i+j\, (\mathrm{mod}\, 2)}$.

Let $\alJ$ be a Jordan superalgebra and let $\mM$ be a Jordan $\alJ$-superbimodule.
Then a Jordan superalgebra $\alE$ is called an \textit{extension of $\alJ$ by $\mM$}, if there exists a short exact sequence of superalgebras 
		\begin{displaymath}\xymatrix@R=0.2mm{0\ar[r]& \mM\ar[r]^{\alpha}&\alE \ar[r]^{\beta}&\alJ \ar[r]&0}.
		\end{displaymath}

Two extensions $\alE$ and $\alE^{'}$ of $\alJ$ by $\mM$ one said to be equivalent, if there exists a homomorphism of superalgebras $\phi:\alE\rightarrow\alE^{'}$, such that the diagram
		\begin{displaymath}
		\xymatrix@R=0.3mm{
		 & & \mathcal{E}\ar[rd]^\beta\ar[dd]^\phi &  & \\
		0\ar[r]& \mM\ar[ru]^{\alpha}\ar[dr]_{\alpha^\prime}& &\alJ \ar[r]&0\\
		& & \mathcal{E}^\prime\ar[ru]_{\beta^\prime} &  &}
		\end{displaymath}
is commutative.

Let $\alJ$ be a Jordan superalgebra and $\mM$ a Jordan $\alJ$-superbimodule,
then $\alE$ is a split extension of $\alJ$ by $\mM$, if the short exact sequence \begin{displaymath}\xymatrix@R=0.2mm{0\ar[r]&\mM\ar[r]^{\alpha}&\alE\ar[r]^{\beta}&\alJ\ar[r]&0}\end{displaymath} admits a decomposition, i.e. there exists a homomorphism  of superalgebras
$\delta:\alJ\longrightarrow\mathcal{E}$ such that $\beta\circ\delta=1_{\alJ}$. For this existence, we observe that there exists a linear mapping $\delta:\alJ\longrightarrow\mathcal{E}$ such that 
$\alE=\alpha(\mM)\oplus\delta(\alJ)$, i.e. is a direct sum of vector spaces. Therefore it is necessary to find some conditions under which $\delta$ satisfies $\delta(ab)-\delta(a)\delta(b)=0$, i.e. $\delta$ is a homomorphism of Jordan superalgebras.

Note that if $\widetilde{a}\in\alJ$, then there exists $a\in\mathcal{E}$, such that $\beta(a)=\widetilde{a}$. Therefore, if $b\in\alE$ is another element such that $\beta(b)=\widetilde{a}$, then there exists $m\in\mM$, such that $\alpha(m)=b-a$.
Then, for  any $a,b\in\alJ$ there is a unique element $m\in\mM$, such that $m=\delta(ab)-\delta(a)\delta(b)$.

So, we define $h:\alJ\times\alJ\longrightarrow\mM$, where  
		\begin{equation}\label{h:definition}
		h(a,b)=\delta(a)\delta(b)-\delta(ab).
		\end{equation} 
  Note that $h$  is a bilinear form, since $\delta$ is linear.
  
Let $\delta^\prime: \alJ \longrightarrow \alE$ be another homomorphism of Jordan superalgebras, such that
 $ \beta \circ \delta^\prime=1_{\alJ}$ and 
$h^\prime:\alJ\times\alJ\longrightarrow \mM$, where 
$h^\prime(a,b)=\delta^\prime(ab)-\delta^\prime(a)\delta^\prime(b) $. 
$h$ and $h^\prime$ are said to be equivalents if there exists a linear mapping 
$\mu:\alJ\longrightarrow \mM$ such that $h(a,b)-h^\prime(a,b)=a\cdot \mu(b)+\mu(a)\cdot b-\mu(ab)$, where $\cdot$ denotes the action of $\alJ$ over $\mM$. In particular, observe that $h(a,b)$ is equivalent to zero, if there exists a linear mapping $\mu$ such that $h(a,b)= a\cdot \mu(b)+\mu(a)\cdot b-\mu(ab)$ (see \cite{Homological-algebra,Jabcoson1968} for details).

The superalgebra $\alJ$ can be identified with  $\delta(\alJ)$ and a superbimodule $\mM$ with $\alpha(\mM)$, thus
 $\alE=\alpha(\mM)\oplus\delta(\alJ)=\mM\oplus\alJ$, and therefore we can 
consider the nonzero multiplication $\ast$ in $\alE$, which is defined as 
$a\ast b=ab+h(a,b)$, $a\ast m=a\cdot m$, $m\ast a=m\cdot a$ 
for every $a,\,b\,\in \alJ$ and $m\in \mM$.
 
The bilinear forms $h$ defined  by \eqref{h:definition} are called \textit{cocycles}. 
The space determined by all cocycles  is said to be \textit{the cocycle space}. We denote this space by
	\begin{equation}\label{zn}
	\mathcal{Z}^2( \alJ,\mM):= \mathcal{Z}^2_{0}(\alJ,\mM)\dot{+}\mathcal{Z}^2_{1}(\alJ,\mM),
	\end{equation}
where 
$\mathcal{Z}^2_{k}(\alJ,\mM)=\left\{h\in \mathcal{Z}^2(\alJ,\mM)\, \mid h(\alJ_i,\alJ_j)\subseteq \mM_{i+j+k}\right\}$ 
for $i,\,j,\,k\in \mathbb{Z}_{2}$. 

The \textit{coboundery space} is the set of all cocycles that are equivalent to the bilinear form zero, and therefore we denote it by 
	\begin{equation}\label{bn}
	\mathcal{B}^2(\alJ,\mM):=\mathcal{B}^2_{0}(\alJ,\mM)\dot{+}\mathcal{B}^2_{1}(\alJ,\mM).
	\end{equation} 
These cocycles are called \textit{cobounderies}. Note that all cobounderies define an extension of 
$\alJ$ by $\mM$, which is isomorphic to the split-null extension $\alE=\alJ\oplus\mM$.
Moreover, the elements $h,\,h^\prime\in \mathcal{Z}^{2}(\alJ,\mM)$ define an equivalent extension, if $h-h^\prime\in \mathcal{B}^{2}(\alJ,\mM)$.

\begin{definition}\label{def:scg}
Let $\alJ$ be a Jordan superalgebra over  $\F$ 
and $\mM$ be a $\alJ$-superbimodule. The \textit{second cohomology group of  $\alJ$ with coefficients of $\mM$} (SCG) is defined as the quotient group of  
	\begin{equation}\label{2gc}
	\mathcal{H}^{2}(\alJ,\mM):=\mathcal{Z}^{2}(\alJ,\mM)/\mathcal{B}^{2}(\alJ,\mM)
	.\end{equation}
\end{definition}

Note that \eqref{2gc} is trivial when the WPT holds for a Jordan superalgebra $\alJ$ and a $\alJ$-superbimodule $\mM$. 

\section{SCG in Jordan superalgebras}\label{SCG:h:functions}

In Section \ref{preliminares}, we have defined the SCG for Jordan superalgebras by \eqref{2gc}. In this section, we deduce some identities that satisfy the bilinear forms $h$ which define the SCG for Jordan superalgebras. These conditions will be used  to calculate the SCG of a Jordan superalgebra which has a non-splitting extension. Namely, we prove the following statement.

\begin{theorem}\label{TeoAlex1}
Let $\alJ$ be a Jordan  superalgebra, let $\mM$ be a $\alJ$-superbimodule, and let
	\begin{equation*}
	\begin{aligned}
	F(a,b,c,d)&:=\fhi{a}{b}{c}{d},\\
	G(a,b,c,d)&:=\fhd{a}{b}{c}{d},
	\end{aligned}
	\end{equation*}
where $h$  are the cocycles that define the SCG for $\alJ$ with coefficients of $\mM$. 
	
Then the following formulas are true for the bilinear forms:
	\begin{equation}\label{hfsj0}
	h(a,b)=(-1)^{\pa\pb}h(b,a)
	\end{equation}
and
\begin{align}
	&F(a,b,c,d)+(-1)^{\pb(\pc+\pd)+\pc\pd}F(a,d,c,b)+(-1)^{\pa(\pb+\pc+\pd)+\pc\pd}F(b,d,c,a)\nonumber\\
	&\qquad=G(a,b,c,d)+(-1)^{\pb\pc}G(a,c,b,d)+(-1)^{\pd (\pc+\pd)}G(a,d,b,c)\label{hfsj1}
	\end{align}
for all $a,\,b,\,c$, $d\in\alJ_0\dot\cup\alJ_1$.
\end{theorem}
\begin{proof}
Let $\mM$ be a $\alJ$-superbimodule of the Jordan superalgebra $\alJ=\alJ_0\dot{+}\alJ_1$, let $\mathcal{E}$ be an extension of $\alJ$ by $\mM$. Then 
it is clear that  $\alE$  is a Jordan superalgebra, and the equalities \eqref{sj1} and \eqref{sj2} hold for the multiplication $\ast$. It is easy to see that \eqref{hfsj0} is obtained from \eqref{sj1}. 
 
To prove \eqref{hfsj1}, observe that $a\ast b=ab+h(a,b)$ by the definition of $\ast$. Besides, 
	\begin{equation}\label{f1}
	((a\ast b)\ast c)\ast d =((ab)c)d+h((ab)c,d)+h(ab,c)\cdot d+(h(a,b)\cdot c)\cdot d
	\end{equation} 
and 
	\begin{equation}\label{f3}
	(a\ast b)\ast (c\ast d)=(ab)(cd)+h(ab,cd)+h(a,b)\cdot (cd)+(ab)\cdot h(c,d).
	\end{equation}
	
Similar to equations \eqref{f1} and \eqref{f3}, we write 
	\begin{equation}\label{f11}
	((a\ast d)\ast c)\ast b,\quad ((b\ast d)\ast c)\ast a,\quad (a\ast c)\ast (b\ast d)\quad\text{and}\quad(a\ast d)\ast (b\ast c).
	\end{equation}
	Further, substituting \eqref{f1}, \eqref{f3} and \eqref{f11} in \eqref{sj2} and using \eqref{hfsj0}, we obtain that  
		\begin{equation*}
		\begin{split}
		\fhc{a}{b}{c}{d},
		\end{split}
		\end{equation*}
		which proves  \eqref{hfsj1}.
\end{proof}


\section{SCG to the Jordan superalgebra $\mathcal{D}_{t}$}\label{SCG:Dt}

In this section, we prove the main theorem of the present paper, which gives the form of the SCG of the Jordan superalgebra $\mathcal{D}_{t}$ using coefficients in $\redt$. Let $\mathcal{D}_{t}=(\F\cdot e_1+\F\cdot e_2)\dot{+}(\F\cdot x+\F\cdot y)$, $t\neq 0$, with nonzero multiplication 
		\begin{equation*}
 		e^{2}_{i} =e_i,\quad e_{i}x=\frac{1}{2}x,\quad e_{i}y=\frac{1}{2}y\quad\text{for}\quad i=1,2\quad\text{and}\quad xy=e_1+te_2.
		\end{equation*}
By the results of \cite{consu-zel}, if $t\neq 0$, then $\mathcal{D}_{t}$ is a simple Jordan superalgebra over $\F$. Besides, in \cite{Faber1} it is proved that an analogue of WPT is valid for the Jordan superalgebras $\alD$, when some conditions are imposed on the irreducible Jordan $\alD$-superbimodules. Therefore, by the results of \cite{Faber1} the SCG of $\mathcal{D}_{t}$ with coefficients in $\redt$, $\mathcal{H}^{2}(\mathcal{D}_{t},\,\redt)$, is not trivial.

To prove the theorem, we have to calculate $\mathcal{H}^{2}(\alD,\redt)$ by means of Theorem \ref{TeoAlex1}. To this end, we suppose that $\redt=(\Fd\we_1+\Fd\we_2)\dotplus(\Fd\wx+\Fd\wy)$ is a regular $\alD$-superbimodule with the isomorphism $\varphi(e_i)=\we_i$, $\varphi(x)=\wx$,  $\varphi(y)=\wy$ $(i=1,2)$ and prove the following two lemmas. 

\begin{lemma}\label{L1}
	$\mathcal{H}^2_{0}(\alD,\redt)=0$, $t\neq 0$.
\end{lemma}

\begin{proof}
By \eqref{zn}, if $h$ is a bilinear form, then $h\in\zfmc{0}{\redt}$ which means that $h(e_i,e_j)$, $h(x,x)$, $h(y,y)$, $h(x,y)$, $h(y,x)\in (\redt)_0$ and $h(e_i,x)$, $h(x,e_i)$, $h(e_i,y)$, $h(y,e_i)\in(\redt)_1$ for $i, j = 1,2$. Therefore, by \eqref{hfsj0} we just have to consider the elements $h(e_i,e_i)$, $h(e_i,e_j)$, $h(e_i,x)$, $h(e_i,y)$, $h(x,x)$, $h(y,y)$  and $h(x,y)$ for $i, j = 1,2$. Thus, we write $h\in(\redt)_0$ as
$h(e_i,e_i)=\alpha_{i1}\widetilde{e}_1 +\alpha_{i2}\widetilde{e}_2$, $h(e_1,e_2)=\beta_1 \widetilde{e}_1 +\beta_2 \widetilde{e}_2$, $h(x,x)=\eta_1 \widetilde{e}_1 +\eta_2 \widetilde{e}_2$, $h(y,y)=\lambda_1\widetilde{e}_1 +\lambda_2\widetilde{e}_2$, $h(x,y)=\omega_1 \widetilde{e}_1 +\omega_2\widetilde{e}_2$ and  $h\in (\redt)_1$ as $h(e_i,x)=\gamma_{ix}\wx+\gamma_{iy}\wy$,		$h(e_i,y)=\theta_{ix}\wx+\theta_{iy}\wy$, where
$\alpha_{ij}$, $\beta_i$, $\eta_i$, $\lambda_i$, $\omega_i$,  $\gamma_{il}$,
$\theta_{il}\in\F$ for $i, j = 1,2$ and $l=x,y$.

Using \eqref{hfsj1}, we proceed to determine the constants $\alpha_{ij}$, $\beta_i$, $\eta_i$, $\lambda_i$, $\omega_i$,  $\gamma_{il}$,
$\theta_{il}\in\F$, for $i, j = 1,2$ and $l=x,y$.
Assuming that $u$ is the odd element of $\alD$ and substituting it in \eqref{hfsj0} we get $h(u,u)=(-1)^{|u||u|}h(u,u)$ which is equivalent to $h(u,u)=-h(u,u)$. Then, we conclude that $h(u,u)=0$. In particular, for $u=x$ we get $h(x,x)=0$. In a similar way, we obtain that also $h(y,y)=0$.

Now, replacing $a=b=c$ by $e_i$ and also $d$ by $e_j$ $(i\neq j)$ in \eqref{hfsj1}, we obtain 
		\begin{align*}
		&\fhi{e_{i}}{e_{i}}{e_{i}}{e_{j}}\\
		&\hskip5mm +2\big(\fhi{e_{i}}{e_{j}}{e_{i}}{e_{i}}\big)\\
		&=\fhd{e_{i}}{e_{i}}{e_{i}}{e_{j}}\\
		&\hskip5mm+2\big(\fhd{e_{i}}{e_{j}}{e_{i}}{e_{i}}\big).
		\end{align*}
By the multiplication in $\alD$ and the action of $\alD$ over $\redt$, we get
		\begin{equation}\label{eqn:01}
		h(e_i,e_j)+h(e_i,e_i)\cdot e_j=h(e_i,e_j)\cdot e_i .
		\end{equation}
Rewritting \eqref{eqn:01}, we conclude 
		\begin{equation}\label{h0:eq1}
		\beta_1\we_1+\beta_2\we_2+\alpha_{ij}\we_j=\beta_{i}\we_i.
		\end{equation}
Assuming that $i=1$ and $j=2$ in \eqref{h0:eq1} and using the linear independence of $\we_i$, $i=1,2$, we obtain that $\beta_2+\alpha_{12}=0$. In a similar way, taking $i=2$ and $j=1$ in \eqref{h0:eq1} we find that $\beta_{1}+\alpha_{21}=0$. 

Further, substituting $a$ by $u$ and $b=c=d$ by $e_i$ in \eqref{hfsj1}, for the odd element $u$ of $\alD$ we get
		\begin{align*}
		&2\left(\fhi{u}{e_i}{e_i}{e_i}\right)\\
		&\hskip5mm +\fhi{e_i}{e_i}{e_i}{u}\\
		&=2\left(\fhd{u}{e_i}{e_i}{e_i}\right)\\
		&\hskip5mm +\fhd{e_i}{e_i}{e_i}{u}.
		\end{align*}
Using the multiplication in $\alD$ and the action of $\alD$ over $\redt$, we obtain that the last equation is equivalent to 
		\begin{equation}\label{h0:eq2}
		2\left(h(u,e_i)\cdot e_i\right)\cdot e_i + \left(h(e_i,e_i)\cdot e_i\right)\cdot u = 2h(u,e_i)\cdot e_i + h(e_i,e_i)\cdot\left( e_i u\right)
		\end{equation}
Putting $u=x$ and $i=1$ in \eqref{h0:eq2}, by the linear independence of $\wx$ and $\wy$ we get 
		\begin{equation}\label{eqn:02}
		\gamma_{1y}=0\quad \text{and}\quad 2\gamma_{1x}=\alpha_{11}-\alpha_{12}.
		\end{equation}
Similarly, substituting $u=x$ and $i=2$ in \eqref{h0:eq2} we get 
		\begin{equation}\label{eqn:03}
		\gamma_{2y}=0\quad\text{and}\quad 2\gamma_{2x}=\alpha_{22}-\alpha_{21}.
		\end{equation}
Also, the substitutions $u=y$ and $i=1$ in \eqref{h0:eq2} give 
		\begin{equation}\label{eqn:04}
		\theta_{1x}=0\quad \text{and}\quad 2\theta_{1y}=\alpha_{11}-\alpha_{12},
		\end{equation}
while the substitutions $u=y$ and $i=2$ in \eqref{h0:eq2} imply 
		\begin{equation}\label{eqn:05}
		\theta_{2x}=0\quad \text{and} \quad 2\theta_{2y}=\alpha_{22}-\alpha_{21}.
		\end{equation}
By the equalities \eqref{eqn:02}, \eqref{eqn:03}, \eqref{eqn:04} and \eqref{eqn:05} we conclude that $\gamma_{ix}=\theta_{iy}$ and $\gamma_{iy}=\theta_{ix}=0$ for $i=1,2$. Therefore, $h(e_i,x)=\gamma_{ix}\wx$ and $h(e_i,y)=\gamma_{ix}\wy$ for $i=1,2$.

Substituting $a=x$,  $b=e_1$, $c=y$ and $d=e_2$ in \eqref{hfsj1}, we get 
\begin{align*}
		&\fhi{x}{e_1}{y}{e_2}\\
		&\hskip4mm +\fhi{x}{e_2}{y}{e_1}\\
		&\hskip8mm -\big(\fhi{e_1}{e_2}{y}{x}\big)\\
		&=\fhd{x}{e_1}{y}{e_2}\\
		&\hskip4mm +\fhd{x}{e_2}{y}{e_1}\\
		&\hskip8mm -\big(\fhd{e_1}{e_2}{y}{x}\big),
		\end{align*}
and simplifying this equality we obtain
		\begin{equation*}
		\begin{split}
		h(xy,e_2)+2(h(x,e_1)\cdot y)\cdot e_2+h(xy,e_1)+2(h(x,e_2)\cdot y)\cdot e_1-2(h(e_1,e_2)\cdot y)\cdot x\\
		=x\cdot(h(y,e_1)+h(y,e_2))+(h(x,e_1)+h(x,e_2))\cdot y+2h(e_1,e_2)\cdot(xy).
		\end{split}
		\end{equation*}
It is easy to see that the calculation of the left and right side of the latter equality based on the linear independence of $\we_1$ and $\we_2$ gives 
		\begin{equation}\label{eqn:06}
		2\gamma_{1x}=\beta_2 +\alpha_{11}\quad \text{and}\quad2\gamma_{2x}=\beta_1+ \alpha_{22}.
		\end{equation}
Similarly, considering all replacements of elements $\alD$ in $\eqref{hfsj1}$ we get formulas \eqref{eqn:02}, \eqref{eqn:03}, \eqref{eqn:04}, \eqref{eqn:05} and \eqref{eqn:06}.

Now, solving the linear equations given by formulas \eqref{h0:eq1} and \eqref{eqn:02}-\eqref{eqn:06}, we obtain that $\gamma_{ix}$, $\alpha_{ii}$,  $\alpha_{ij}$, $\beta_{j}=-\alpha_{ij}$ $(i\neq j)$ and $\omega_{i}$ $(i,j=1,2)$ are  nonzero constants.
	
Observe that, if $h\in \zfmc{0}{\textrm{Reg}\alD}$, then $h(x,x)=h(y,y)=0,$
		\begin{align*}
		h(e_1, e_1)&=\alpha_{11}\we_{1}+\alpha_{12}\we_2,\quad			&h(e_2,e_2)&=\alpha_{21}\we_{1}+\alpha_{22}\we_2,\\
		h(e_1, e_2)&=-\alpha_{21}\we_{1}-\alpha_{12}\we_2,\quad  		&h(x,y)&=\omega_1\we_{1}+\omega_2\we_2,\\
		h(e_1,x)&=\left(\frac{\alpha_{11}-\alpha_{12}}{2}\right)\wx,\qquad 		&h(e_1,y)&=\left(\frac{\alpha_{11}-\alpha_{12}}{2}\right)\wy,\\
		h(e_2,x)&=\left(\frac{\alpha_{22}-\alpha_{21}}{2}\right)\wx,\qquad 		&h(e_2,y)&=\left(\frac{\alpha_{22}-\alpha_{21}}{2}\right)\wy.
		\end{align*}
Consider that $h_i\in \zfmc{0}{\redt}$, $i=1,\hdots,6$, such that the nonzero images of the bilinear forms are 		
\begin{align*}
		h_1(e_1,e_1)&=\we_1, 			&h_1(e_1,x)&=\frac{\wx}{2},		&h_1(e_1,y)&=\frac{\wy}{2}, &&\\
		h_2(e_1,e_1)&=\we_2, 			&h_2(e_1,e_2)&=-\we_2,			&h_2(e_1,x)&=-\frac{\wx}{2}, &h_2(e_1,y)&=-\frac{\wy}{2},\\
		h_3(e_2,e_2)&=\we_1,			&h_3(e_1,e_2)&=-\we_1,			&h_3(e_2,x)&=-\frac{\wx}{2},&h_3(e_2,y)&=-\frac{\wy}{2},\\
		h_4(e_2,e_2)&=\we_2,			&h_4(e_2,x)&=\frac{\wx}{2},		&h_4(e_2,y)&=\frac{\wy}{2},&&\\
		h_5(x,y)&=\we_1,				&h_6(x,y)&=\we_2.				&& &&
		\end{align*}
These bilinear forms generate $\zfmc{0}{\textrm{Reg}\alD}$, for $i= 1,\hdots, 6 $. 

Now, let $h \in \bfmc{0}{\redt}$. Then for all $a,b\in\alD$, $h$ is of the form \begin{equation}\label{eqn_cobounderiesz0}
h(a,b)=-\mu(ab)+a\cdot\mu(b)+\mu(a)\cdot b,
\end{equation} where $\mu:\alD\rightarrow\redt$ is an even linear mapping. So, we write  $\mu(e_i)=\lambda_{i1}\we_1+\lambda_{i2}\we_2$, $\mu(x) =\lambda_{xx}\wx+\lambda_{xy}\wy$ and $\mu(y)=\lambda_{yx}\wx+\lambda_{yy}\wy$,
where $\lambda_{ij},\,\lambda_{lk}\in\F$, 
for $i, j=1,2$ and $l,k=x,y$. Now, considering $h(a,b)$ for all $a,b\in\alD$ in \eqref{eqn_cobounderiesz0} we obtain 
		\begin{equation}\label{eqn:07}
		\begin{split}
		h(e_i,e_i)&=-\mu(e_i^2)+e_{i}\cdot\mu(e_i)+\mu(e_i)\cdot e_i\\
		&=-\mu(e_i)+e_i\cdot\mu(e_i)+\mu(e_i)\cdot e_i\\
		&= -(\lambda_{i1}\we_1+\lambda_{i2}\we_2)+2\lambda_{ii}\we_i.
		\end{split}
		\end{equation}
Setting here $i=1$, we get $h(e_1,e_1)=\alpha_{11}\we_1+\alpha_{12}\we_2=\lambda_{11}\we_1-\lambda_{12}\we_2$. By the linear independence of $\we_1$ and $\we_2$ we have 
\begin{equation}\label{x}
\lambda_{11}=\alpha_{11}\quad\text{and}\quad\lambda_{12}=-\alpha_{12}.
\end{equation}
Similarly, taking $i=2$ in \eqref{eqn:07} we get $h(e_2,e_2)=\alpha_{21}\we_1+\alpha_{22}\we_2=-\lambda_{21}\we_1+\lambda_{22}\we_2$. Thus,
\begin{equation}\label{y}
\lambda_{21}=-\alpha_{21}\quad\text{and}\quad\lambda_{22}=\alpha_{22}.
\end{equation} Further, observe that \begin{equation*}
		\begin{split}
		h(x,y)&=-\mu(xy)+x\cdot \mu(y)+\mu(x)\cdot y\\
			&=-\mu(e_1)-t\mu(e_2)+\lambda_{yy}(\we_1+t\we_2)+\lambda_{xx}(\we_1+t\we_2)\\
			&=-\mu(e_1)-t\mu(e_2)+(\lambda_{xx}+\lambda_{yy})(\we_1+t\we_2),
		\end{split}
		\end{equation*}
and hence $h(x,y)=\omega_{1}\we_1+\omega_{2}\we_2 =(-\alpha_{11}+\alpha_{12}t)\we_1+(\alpha_{12}-\alpha_{22}t)\we_2+(\lambda_{xx}+\lambda_{yy})(\we_1+t\we_2)$. Therefore, by the linear independence of $\we_i$ $(i=1,2)$ we conclude that 
\begin{equation}\label{eqn:010}
\omega_1=-\alpha_{11}+\alpha_{12}t+\lambda_{xx}+\lambda_{yy}\quad \text{and}\quad
\omega_2=\alpha_{12}-\alpha_{22}t+(\lambda_{xx}+\lambda_{yy})t.
\end{equation}
In a similar way, considering $h(a,b)$ for all $a$, $b\in\alD$  in \eqref{eqn_cobounderiesz0} we obtain \eqref{x}, \eqref{y} and \eqref{eqn:010}. Besides, solving the linear equations given by \eqref{x}, \eqref{y} and \eqref{eqn:010} we get 
		\begin{align}
		\lambda_{11}&=\alpha_{11}, \quad\lambda_{22}=\alpha_{22},\quad &\lambda_{12}&=-\alpha_{12},\quad\lambda_{21}=-\alpha_{21},\nonumber\\
		\omega_1 &=\lambda_{xx}+\lambda_{yy}-\alpha_{11}+\alpha_{12}t, &\omega_2&=\alpha_{12}-\alpha_{22}t+(\lambda_{xx}+\lambda_{yy})t.\label{eqn:0011}
		\end{align}
Therefore, by $\eqref{eqn:0011}$ we conclude that if $h\in\zfmc{0}{\redt}$, then $h\in\bfmc{0}{\redt}$. Consequently, $\zfmc{1}{\redt}=\bfmc{1}{\redt}$. Thus
		\begin{equation*}
		\hfmc{0}{\redt}=\zfmc{0}{\redt}/\bfmc{0}{\redt}
		\end{equation*}
is isomorphic to $0$ for all $t\neq 0$.    
\end{proof}

\begin{lemma}\label{L2}
 $\mathcal{H}^2_{1}(\alD,\redt)=\F^2$, $t\neq 0$.
\end{lemma}

\begin{proof}
By \eqref{zn}, if $h$ is a bilinear form, then $h\in\zfmc{1}{\redt}$ which means that $h(e_i,e_j)$,  $h(x,x)$, $h(y,y)$, $h(x,y)$, $\,h(y,x)\in(\redt)_1$ and $h(e_i,x)$, $h(x,e_i)$, $h(e_i,y)$, $h(y,e_i)\in(\redt)_0$ for $i, j = 1,2$. Therefore, by \eqref{hfsj0}
we just have to consider elements $h(e_i,e_i)$, $h(e_i,e_j)$, $h(e_i,x)$, $h(e_i,y)$, $h(x,x)$, $h(y,y)$  and $h(x,y)$ for $i$, $j=1,2$. Thus, we  write bilinear forms $h\in(\redt)_1$ as $h(e_i,e_i)=\alpha_{ix}\wx +\alpha_{iy}\wy$, $h(e_1,e_2)=\beta_x \wx +\beta_y \wy$, $h(x,x)=\eta_x \wx +\eta_y \wy$, $h(y,y)=\lambda_x\wx +\lambda_y\wy$, $h(x,y)=\omega_x \wx +\omega_y\wy$ and $h\in(\redt)_0$ as 	$h(e_i,x)=\gamma_{i1}\we_1+\gamma_{i2}\we_2$, $h(e_i,y)=\theta_{i1}\widetilde{e}_1+\theta_{i2}\widetilde{e}_2$, where $\alpha_{ik}$, $\beta_k$, $\eta_k$, $\lambda_k$, $\omega_k$, $\gamma_{ij}$,
$\theta_{ij} \in\F$, for $i, j = 1,2$ and $k= x, y$. 

Using \eqref{hfsj1}, we proceed to determine $\alpha_{ik}$, $\beta_k$, $\eta_k$, $\lambda_k$, $\omega_k$, $\gamma_{ij}$, $\theta_{ij} \in\F$, for $i, j = 1,2$ and $k= x, y$. Assuming that $u$ is odd element of $\alD$ and substituting in \eqref{hfsj0} we have that $h(u,u)=(-1)^{|u||u|}h(u,u)$. Thus $h(u,u)=-h(u,u)$. Then, we conclude that $h(u,u)=0$. In particular, $h(x,x)=0$ and $h(y,y)=0$.

Now, substituting $a=b=c=d$ by $e_{i}$ in \eqref{hfsj1}, we  get $(h(e_i,e_i)\cdot e_i)\cdot e_i=e_i\cdot h(e_i,e_i)$ and then, we conclude that $\alpha_{ix}\wx+\alpha_{iy}\wy=0$. By the linear independence  of $\wx$ and $\wy$, we get $\alpha_{ix}=\alpha_{iy}=0$ for $i=1,2$. Therefore, $h(e_i,e_i)=0$ for $i=1,2$. Further, substituting $a$ by $u$, $b=d$ by $e_i$ and $c$ by $e_j$ $(i\ne j)$ and $u$ an odd element in \eqref{hfsj1}, we get
		\begin{align*}
		&2\big(\fhi{u}{e_i}{e_j}{e_i}\big)\\
		&\hskip5mm +\fhi{e_i}{e_i}{e_j}{u}\\
		& =2\big(\fhd{u}{e_i}{e_j}{e_i}\big)\\
		&\hskip5mm +\fhd{e_i}{e_i}{e_j}{u}.
		\end{align*}	
Applying $h(e_i,e_i)=0$ and the action of $\alD$ over $\redt$, we get
		\begin{equation}\label{eqn:11}
		h(e_i,e_j)\cdot u=u\cdot h(e_i,e_j).
		\end{equation}
Putting $u=x$ in \eqref{eqn:11}, we obtain $\beta_y=0$. Similarly, if $u=y$ in \eqref{eqn:11}, we find that $\beta_x=0$. Therefore, we conclude $h(e_1,e_2)=0$. Moreover, substituting  $a$ by $u$, $b=c$ by $e_i$ and $d$ by $e_j$ $(i\ne j)$ with $u$ an odd element in \eqref{hfsj1}, we get
		\begin{align*}
&\fhi{u}{e_i}{e_i}{e_j}\\
&\hskip4mm +\fhi{u}{e_j}{e_i}{e_i}\\
&\hskip8mm +\fhi{e_i}{e_j}{e_i}{u}\\
& =\fhd{u}{e_i}{e_i}{e_j}\\
&\hskip4mm +\fhd{u}{e_j}{e_i}{e_i}\\
&\hskip8mm +\fhd{e_i}{e_j}{e_i}{u},
\end{align*}
using the multiplication and the action, we obtain that the last equation is equivalent to
		\begin{equation}\label{iiuj1}
		h(u,e_j)+2h(u,e_i)\cdot e_j+ 2h(u,e_i)\cdot e_i= h(u,e_i).
		\end{equation}
Assuming $u=x$ in \eqref{iiuj1}, we obtain $h(x,e_j)+2h(x,e_i)\cdot e_j+ 2h(x,e_i)\cdot e_i= h(x,e_i)$. Further, substituting $i=1$ and $j=2$ in the last equality and using the linear independence of $\wx$ and $\wy$, then	\begin{equation}\label{eqn1:1}
		\gamma_{11}+\gamma_{21}=0\quad \text{and}\quad \gamma_{22}+\gamma_{12}=0.
		\end{equation} 
In the same way, putting $u=y$ in \eqref{iiuj1}, we obtain $h(y,e_j)+2h(y,e_i)\cdot e_j+ 2h(y,e_i)\cdot e_i= h(y,e_i)$. Assuming $i=1$ and $j=2$ in this equation, by the linear independence of $\wx$ and $\wy$ we find 
		\begin{equation}\label{eqn1:2}
		\theta_{11}+\theta_{21}=0\quad\text{and}\quad \theta_{22}+\theta_{12}=0.
		\end{equation}
Further, substituting $a$ by $x$, $b$ by $y$ and $c=d$  by $e_i$ in \eqref{hfsj1}, then	
		\begin{align}
        &4(h(x,e_i)\cdot e_i)\cdot y+h(x,y)-4(h(y,e_i)\cdot e_i)\cdot x\nonumber\\ 
        &\hskip4mm =4h(x,y)\cdot e_i+2x\cdot h(e_i,y)-2y\cdot h(e_i,x). \label{eqn:123}
		\end{align}
Assuming that $i=1$ in \eqref{eqn:123} we conclude
		\begin{equation}\label{eqn1:3}
		\omega_x=-\theta_{12}-3\theta_{11}\quad \text{and}\quad \omega_y=3\gamma_{11}+\gamma_{12}.
		\end{equation}
Similarly, putting $i=2$ in \eqref{eqn:123} by the linear independence of $\wx$ and $\wy$, we get
		\begin{equation}\label{eqn1:4}
		\omega_x=-\theta_{21}-3\theta_{22}\quad \text{and}\quad \omega_y=3\gamma_{22}+\gamma_{21}.
		\end{equation}
Similarly, considering all replacements of elements $\alD$ in $\eqref{hfsj1}$ we get formulas \eqref{eqn1:1}-\eqref{eqn1:4}.

Now, solving  the linear equation system  giving by \eqref{eqn1:1}-\eqref{eqn1:4} we get that nonzero constants are 
	$ \gamma_{11}=\gamma_{22}=-\gamma_{21}=-\gamma_{12}$,  
		$\theta_{11}=\theta_{22}=-\theta_{21}=-\theta_{12}$, $\omega_{y}=2\gamma_{11}$ and      $\omega_{x}=-2\theta_{11}$.
Observe that $h\in \zfmc{1}{\redt}$, then $h(e_i,e_j)=h(e_i,e_i)=h(x,x)=h(y,y)=0$ for $i,j=1,2$ and $h(e_1, x)=\gamma_{11}(\we_{1}-\we_2)$, $h(e_2,x)=\gamma_{11}(-\we_{1}+\we_2)$,
$h(e_1, y)=\theta_{11}(\we_{1}-\we_2)$, $h(e_2,y)=\theta_{11}(-\we_{1}+\we_2)$ and $h(x,y)=-2\theta_{11}\wx+2\gamma_{11}\wy$.

Consider that $h_i\in \zfmc{1}{\redt}$, $i\in\left\{1,2\right\}$ such that the nonzero images of the bilinear forms are
		\begin{align*}
		&h_1(e_1,x)= \we_{1}-\we_2=-h_1(e_2,x),&h_1(x,y)&=2\wy,\\
		&h_2(e_1,y)= \we_{1}-\we_2=-h_2(e_2,y),&h_2(x,y)&=-2\wx.
		\end{align*}
These bilinear forms generate $\zfmc{1}{\redt}$ for $i= 1, 2$. Now, let $h \in$$\bfmc{1}{\redt}$. If  also $h\in\zfmc{1}{\redt}$ then for all $a,b\in\alD$, 
\begin{equation}\label{eqn:b10}
    h(a,b)=-\mu(ab)+a\cdot\mu(b)+\mu(a)\cdot b
\end{equation} where $\mu:\alD\rightarrow\redt$ is an odd linear mapping. Let $\mu(e_i)=\lambda_{ix}\wx+\lambda_{iy}\wy$, $\mu(x)=\lambda_{x1}\we_1+\lambda_{x2}\we_2$ and $\mu(y)=\lambda_{y1}\we_1+\lambda_{y2}\we_2$, where  $\lambda_{ik},\, \lambda_{ki}\in\F$, for $i=1,2$ and $k=x,y$. Now, we calculate  $h(a,b)$ equivalent to bilinear for zero for all $a,b\in\alD$. Substituting  $a$ by $e_i$ and $b$ by $u$ an odd element  in \eqref{eqn:b10} we obtain
		\begin{align}
		h(e_i,u)&=-\mu(e_i u)+e_i\cdot\mu(u)+\mu(e_i)\cdot u\nonumber\\
				&=-\frac{\mu(u)}{2}+e_i\cdot\mu(u)+\mu(e_i)\cdot u. \label{eqn:b00}
		\end{align}
Similarly to \eqref{eqn:b00}, we get 
		\begin{equation}\label{eqn:b0}
h(u,e_i)=-\frac{\mu(u)}{2}+u\cdot\mu(e_i)+\mu(u)\cdot e_i.
\end{equation} By $h(e_i,u)=h(u,e_i)$ and equalities $\eqref{eqn:b00}$ and $\eqref{eqn:b0}$, it is easy to see that 
\begin{equation}\label{eqn:b01}
u\cdot \mu(e_i)=\mu(e_i)\cdot u.
\end{equation} Assuming $u=x$ in \eqref{eqn:b01} we conclude that $\lambda_{iy}=0$, $i=1,2$. Similarly, with $u=y$, we get  $\lambda_{ix}=0$, $i=1,2$. Therefore, $\mu(e_i)=0$, $i=1,2$. So, we write 		\begin{equation}\label{eqn:b11}
h(e_i,u)=-\frac{\mu(u)}{2}+e_i\cdot\mu(u)
\end{equation}
Let $i=1$ and $u=x$ in \eqref{eqn:b11}, we get $h(e_1,x)=-\frac{1}{2}\big(\lambda_{x1}\we_1+\lambda_{x2}\we_2\big)+\lambda_{x1}\we_1$,
which is equivalent to $\gamma_{11}(\we_1-\we_2)=-\frac{1}{2}\big(\lambda_{x1}\we_1+\lambda_{x2}\we_2\big)+\lambda_{x1}\we_1$.
By the linear independence of $\we_1$ and $\we_2$, we conclude 
$2\gamma_{11}=\lambda_{x1}$ and $2\gamma_{11}=\lambda_{x2}$. Similarly, putting $i=2$ and $u=x$ in \eqref{eqn:b11}, we get	$2\gamma_{11}=\lambda_{x1}$ and $2\gamma_{11}=\lambda_{x2}$. Analogously with $u=y$ in \eqref{eqn:b11}, we  obtain
		$2\theta_{11}=\lambda_{y1}$ and $2\theta_{11}=\lambda_{y2}$. Therefore,
		\begin{equation}\label{eqn:2}
		\lambda_{x1}=\lambda_{x2}\quad \text{and}\quad \lambda_{y1}=\lambda_{y2}.
		\end{equation} 
Further, let $u$ be an odd element in $\alD$, susbtituting $a=b$ by $u$ in \eqref{eqn:b01} we obtain
		\begin{align}
		h(u,u)	&=-\mu(u^2)+u\cdot\mu(u)+u\cdot\mu(u)\nonumber\\
				&=u\cdot \mu(u)+\mu(u)\cdot u.\label{eqn:b3}
		\end{align}
Replacing $u=x$ in \eqref{eqn:b3}, we get $(\lambda_{x1}+\lambda_{x2})\wx=0$. Then by \eqref{eqn:2} in this equation we conclude $\lambda_{xi}=0$, $i=1,2$. Therefore $\mu(x)=0$. Analogously with $u=y$ in \eqref{eqn:b3} by \eqref{eqn:2}, we obtain $\mu(y)=0$. 

In conclusion,  $\mathcal{B}_{1}^{2}\big(\alD,\redt\big)=0$. Consequently,  	
	\begin{equation*}
	\mathcal{H}_{1}^{2}\big(\alD,\redt)=\zfmc{1}{\redt}/\mathcal{B}_{1}^{2}\big(\alD,\redt\big)
	\end{equation*} is isomorphic to $\F^2$.
\end{proof}

Now, we proof the main result of this paper.
\begin{theorem}\label{Main:theorem}
	Let $\alD$ be the Jordan superalgebra, $t\neq 0$. Then 
	$$\mathcal{H}^2(\alD,\redt)= 0\dot{+}\F^2.$$
\end{theorem}
\begin{proof}
	The proof holds by Lemmas \ref{L1}, \ref{L2} and Definition \ref{def:scg}.
\end{proof}

\begin{corollary}
	Let $\mathcal{M}_{1\mid 1}(\F)^{(+)}$ Jordan superalgebra \cite{FaberWPT,consu-zel}, then
	\begin{equation*}
	\mathcal{H}^2\big(\mathcal{M}_{1\mid 1}(\F)^{(+)},\,\mathrm{Reg}\,\mathcal{M}_{1\mid 1}(\F)^{(+)}\big)= 0\dot{+}\F^2.
	\end{equation*}
\end{corollary}

\begin{proof}
	The proof follows from Theorem \ref{Main:theorem} and the isomorphism $\mathcal{D}_{-1}\cong \mathcal{M}_{1\mid 1}(\F)^{(+)}$.
\end{proof}


\bibliographystyle{amsplain}

\end{document}